\documentclass{article}[12pt]
\usepackage{amsmath,amsfonts,amsthm}
\usepackage{amssymb}

\usepackage{epsfig}

\usepackage{graphics}
\usepackage{graphicx}

\newcommand{\beql}[1]{\begin{equation}\label{#1}}
\newcommand{\eeql}{\end{equation}}
\newcommand{\eqn}[1]{(\ref{#1})}

\newcommand{\R}{\mathbb{R}}
\newcommand{\pr}{\mathbb{P}}
\newcommand{\E}{\mathbb{E}}

\newcommand{\Z}{\mathbb{Z}}
\newcommand{\barZ}{\bar{\Z}}

\newtheorem{thm}{Theorem}
\newtheorem{lem}[thm]{Lemma}
\newtheorem{prop}[thm]{Proposition}

\begin{document}

\title{Large number of queues in tandem:\\ Scaling properties under back-pressure algorithm}

\author
{
Alexander L. Stolyar \\
Bell Labs, Alcatel-Lucent\\
600 Mountain Ave., 2C-322\\
Murray Hill, NJ 07974 \\
\texttt{stolyar@research.bell-labs.com}
}

\date{October 30, 2009}

\maketitle

\begin{abstract}

We consider a system with $N$ unit-service-rate queues in tandem,
with exogenous arrivals of rate $\lambda$ at queue $1$,
under a back-pressure (MaxWeight) algorithm: service at queue $n$ is blocked
unless its queue length is greater than that of next queue $n+1$.
The question addressed is how steady-state queues scale as $N\to\infty$.
We show that the answer depends on whether $\lambda$ is below or above the critical value $1/4$:
in the former case queues remain uniformly stochastically bounded,
while otherwise they grow to infinity.

The problem is essentially reduced to the behavior of the system with infinite number 
of queues in tandem, which is studied using tools from interacting particle systems 
theory. In particular, the criticality of load $1/4$ is closely related to the
fact that this is the maximum possible flux (flow rate) of a stationary totally asymmetric simple exclusion
process.

\end{abstract}


\noindent
{\em Abbreviated Title:} Queue scaling under back-pressure



\section{Introduction}
\label{sec-intro}

In this paper we consider a system with $N$ single-server queues in tandem,
with exogenous Poisson arrival process at queue $1$ and customers leaving after service in queue $N$.
All service times have exponential distribution with unit mean.
The waiting room in each queue is unlimited, but 
the system operates under the following back-pressure policy: service of queue
$n$ is blocked (stopped) unless its queue length is greater than that of ``next-hop'' queue $n+1$.
The system is stable as long as input rate
$\lambda < 1$. The main question we address is: Given load $\lambda<1$ is fixed,
how steady-state queue lengths increase (scale) as $N \to \infty$.

Our main motivation comes from the fact that general {\em back-pressure} (BP) policies
(sometimes also called {\em MaxWeight}),
originally introduced by Tassiulas and Ephremides \cite{tas92}
and received much attention in the literature (cf. \cite{St2004gpd,DaiLin} for recent reviews),
are very attractive for application in communication and service networks.
This is due to the adaptive nature of BP policies -- they
can ensure maximum possible network throughput,
without a priori knowledge of flow input rates.
Key mechanism in BP policies, giving it the adaptivity (and the name),
is that the ``priority'' of a traffic flow $f$
at a given network node $n$ is ``proportional'' to the difference between flow $f$ queue
lengths at node $n$
and the next node on the route; in particular, unless this queue-differential is positive,
flow $f$ service at $n$ can be blocked. This mechanism, however, has a drawback, namely it can lead 
to a large queue build-up along a flow route, since, roughly speaking, the queue ``needs'' to
increase as we move back from a flow destination node to its source. Such ``bad scaling'' behavior of 
BP algorithm is emphasized in \cite{bss2009_fifo}, and is not very surprising (see also
our Proposition~\ref{prop-simple-bound} and bound \eqn{eq-mean-growth}). 
One approach to mitigate this scaling problem
in practical systems is proposed in \cite{bss2007_royal,bss2009_fifo}, where it is suggested 
to ``run'' BP algorithm on virtual queues as opposed to physical ones.

It is of interest to understand fundamental scaling properties of BP algorithm. 
In this paper, we address this problem for 
a simple system -- a single flow served by $N$ queues in tandem. 
We show that, perhaps somewhat
surprisingly, the scaling behavior of BP is {\em not always} bad. Namely, if load
is below some critical level, $\lambda<1/4$ for our model,
all queues remain uniformly stochastically bounded for all $N$.
(In fact, we show that the stochastic bound has exponentially decaying tail;
see Theorem~\ref{th-main}(i).)
When $\lambda>1/4$, the queues increase to infinity with $N$ (see Theorem~\ref{th-main}(ii)).

The problem of asymptotic behavior of queues as $N\to\infty$ is essentially
reduced (see Proposition~\ref{prop-ordering})
to the behavior of the system with {\em infinite} number of queues in tandem.
Such infinite-tandem model is within
the framework of interacting particle systems \cite{Liggett-book,Liggett-book99} -- methods and
results of the corresponding theory will be our main tools.
As we will see, the criticality of load $1/4$ is closely related to the fact that
$1/4$ is the maximum possible flux (average flow rate) of a stationary {\em totally asymmetric
simple exclusion process} (TASEP). In the subcritical case, $\lambda<1/4$, a stationary
regime exists such that only a finite (random) number of ``left-most'' queues 
can be greater than $1$,
the rest of the queues have at most one customer and the process ``there'' behaves like TASEP
(see Theorems~\ref{queue-exp-tail} and \ref{busy-interval-exp-tail}).
In the supercritical case, $\lambda>1/4$, each queue grows without bound with time
(see Theorem~\ref{queue-unbounded}).

Infinite series of queues in tandem is a much studied model in the literature
(cf. \cite{Martin2002} and references therein), under a variety of assumptions. 
In particular, \cite{Martin2002} studies an infinite system under a class of blocking 
policies (which are different from the BP policy in our model), where blocking is caused
by limited waiting space (finite buffer) in the queues; there, the  phenomenon of critical 
load, below which the system is stable, also exists.

The rest of the paper is organized as follows. Section~\ref{sec-model} presents the formal model 
and main result. The ``reduction'' of our problem to the behavior of infinite system, 
and basic properties of the latter, are given in Section~\ref{sec-basic}. The subcritical and 
supercritical load cases are treated in Sections~\ref{sec-subcritical} and \ref{sec-supercritical},
respectively.  In Section~\ref{sec-gen-input} we remark on more general
input processes.

\section{Formal Model and Main Result}
\label{sec-model}

Consider a series of $N$ servers (sites), numbered $1,\ldots,N$, each with unlimited queueing room.
A Poisson flow of customers, of rate $\lambda>0$, arrives at server 1, and each customer
has to be served consecutively by the series of servers, from 1 to N; after service by $N$-th server,
a customer leaves the system. 
The service time of any customer at each server is exponentially distributed, 
with mean value $1$; all service times are independent of each other and of the input process.
To be specific, assume that each site serves customers in first-come-first-serve order -- given
Markov assumptions, this will not limit generality of results.

Let us denote by $Q_n^{(N)}(t)$, $n=1,\ldots,N$, $t\ge 0$, the queue length at $n$-th server at time $t$.
The superscript $N$ indicates the number of servers, which will be the parameter we vary.

Consider the following {\em back-pressure} (BP) algorithm: site $n$ is actually serving
a customer (which is the head-of-the-line customer
from its queue) at time $t$ if and only if $Q_n^{(N)}(t) > Q_{n+1}^{(N)}(t)$.
(We use convention that $Q_{N+1}^{(N)}(t) \equiv 0$.) 
In other words, the service at site $n$
is blocked unless the queue at site $n+1$ is smaller.

The random process $(Q^{(N)}(t) \equiv (Q_n^{(N)}(t), ~n=1,2,\ldots,N), ~t\ge 0)$, 
describing evolution of the queues, is a countable irreducible continuos-time Markov chain.
Stability of this process - ergodicity of the Markov chain - is guaranteed 
under condition $\lambda<1$ -- this follows from well-known properties 
of BP algorithms (cf. \cite{DaiLin}). Therefore the unique stationary distribution exists;
we denote by $Q^{(N)}(\infty)$ a random system state in the stationary regime.

The question we address is how
steady-state queues $Q_n^{(N)}(\infty)$ grow (scale) as $N \to \infty$,
more specifically whether or not they remain stochastically bounded.
It is easy to observe that if condition 
\beql{eq-bound-tail}
Q_n^{(N)}(t)+1 \ge \max_{k>n}Q_k^{(N)}(t), ~\forall n\ge 1,
\end{equation}
holds for $t=t_0$, then it holds for all $t\ge t_0$ as well;
in particular, it does hold in the stationary regime. 
Thus, $Q_1^{(N)}(\infty)+1$ is an upper bound on
all queues. Therefore, we can concentrate on the question of whether or not
$Q_1^{(N)}(\infty)$ remains stochastically bounded as $N\to\infty$.
(If it does not, it is easy to see that, for any $n$,
$Q_n^{(N)}(\infty)$ goes to infinity in probability.) We show that the answer depends
on whether or not the input rate $\lambda$ is below or above the critical
value $1/4$. Namely,
our main result is the following
\begin{thm}
\label{th-main}
(i) If $\lambda<1/4$, then there exist $C_1>0$ and $C_2>0$ such that, uniformly on $N$,
\beql{eq-exp-bound-main}
\pr\{Q_1^{(N)}(\infty) > r\} \le C_1 e^{-C_2 r}.
\end{equation}
(ii) If $\lambda>1/4$, then $Q_1^{(N)}(\infty) \to \infty$ in probability as $N\to\infty$.
\end{thm}
Statements (i) and (ii) will follow from Theorems~\ref{queue-exp-tail}
and \ref{queue-unbounded}, respectively, which concern with the corresponding
infinite-tandem system.

Note that the fact that tight uniform bound \eqn{eq-exp-bound-main} {\em cannot} possibly
hold for all $\lambda<1$ is fairly obvious.
Observe that
in the stationary regime the average rate at which customers move from site $n$ to $n+1$, for any $n\le N$, 
is $\lambda$. Therefore, we have
\begin{prop}
\label{prop-simple-bound}
If $\lambda<1$, then for any $n =1,\ldots,N$,
$$
\pr\{Q_n^{(N)}(\infty)>Q_{n+1}^{(N)}(\infty)\} = \lambda.
$$
Then, using \eqn{eq-bound-tail},
$$
\E Q_n^{(N)}(\infty) - \E Q_{n+1}^{(N)}(\infty) \ge 2\lambda-1, ~n=1,\ldots,N.
$$
\end{prop}
Thus, in the case $\lambda>1/2$, we have at least linear growth of the first queue expected value:
\beql{eq-mean-growth}
\E Q_1^{(N)}(\infty) \ge (2\lambda-1)N.
\end{equation}

\section{Basic Properties of Infinite-Tandem Queues under Back-pressure algorithm}
\label{sec-basic}

Consider a system just like the one in Section~\ref{sec-model}, except there is an infinite
number of servers (sites) in tandem, indexed $n=1,2,\ldots$. Arriving customers never leave --
they just keep moving from site to site, to the ``right''. 
We denote by $Q_n(t)$, $n=1,2,\ldots$, $t\ge 0$, the queue length at site $n$ at time $t$,
by $Q(t) \equiv (Q_n(t), ~n=1,2,\ldots)$ the state of the entire system at time $t$.
It will be convenient to assume that the phase space for each queue (site) $n$ state
$Q_n$ is the compact set $\barZ_+ = \Z_+ \cup \{\infty\}$ with $\Z_+$ being the set
of non-negative integers and with metric, e.g., $|e^{-i}-e^{-j}|$. 
(We use convention $\infty - 1 = \infty$, so that the back-pressure algorithm is well-defined.)
The state space of Markov process
$(Q(t), ~t\ge 0)$ is $\barZ_+^S$, $S=\{1,2,\ldots\}$,
with product topology; the process is formally defined 
within the framework of interacting particle systems
(cf. Section I.3 of \cite{Liggett-book}, specifically Theorem I.3.9).
We will use some sligtly abusive notations: $Q(\cdot)$ for the process 
$(Q(t), ~t\ge 0)$, and $Q = (Q_n, ~n=1,2,\ldots) \in \barZ_+^S$ for elements of
the phase space; and will denote $\|Q\| = \sum_n Q_n$. 

Throughout the paper we will also use the following representation 
of process $Q(\cdot)$, which is standard for Markov
interacting-particle systems and is convenient for coupling (cf. \cite{Liggett-book,Liggett-book99})
of different versions of the processes. 
First, assume
that there is a site associated with each
integer $n\in \Z$, not just positive $n$. (We will use this convention
in Section~\ref{sec-subcritical}.)
The underlying probability space 
is such that there is a unit rate Poisson processes $\Pi_n$ for each site
$n$; these processes are independent from each other and of the input flow Poisson
process. Then, if time $\tau \ge 0$ is a point of the Poisson process $\Pi_n$, 
a customer moves from queue $n$ to queue $n+1$
at $\tau$ if $Q_n(\tau-) > Q_{n+1}(\tau-)$, otherwise 
the move is suppressed.

The finite system of Section~\ref{sec-model},
with $N$ sites,
will be viewed as an infinite one, but with the modification that any customer
reaching site $N+1$ is immediately removed from the system, and with $Q_n(t)\equiv 0$
for $n>N$ (and $n\le 0$).

It is easy to check that an analog of \eqn{eq-bound-tail} holds 
for the infinite system as well. Namely, if
\beql{eq-bound-tail-inf}
Q_n(t)+1 \ge \sup_{k>n}Q_k(t), ~\forall n\ge 1,
\end{equation}
holds for $t=t_0\ge 0$, then it holds for all $t\ge t_0$ as well.

Now we state basic monotonicity properties of the infinite system.
(They are easily proved by contradiction, using coupling on the common probability
space defined above.)
The inequality $Q \le Q'$ is understood component-wise; the order relation
$Q \preceq Q'$ means that both $\|Q\|$ and $\|Q'\|$ are finite and 
$$
\sum_{k\ge n} Q_k \le \sum_{k\ge n} Q'_k, ~~\forall n.
$$

\begin{lem}
\label{lem1}
If $Q(0) \le Q'(0)$ [respectively, $Q(0) \preceq Q'(0)$],
then the processes $Q(\cdot)$ and $Q'(\cdot)$ can be coupled 
so that $Q(t)\le Q'(t)$ 
[respectively, $Q(t) \preceq Q'(t)$]
for all $t\ge 0$.
\end{lem}

As a corollary of Lemma~\ref{lem1}, we obtain the following

\begin{prop}
\label{prop-ordering}
Consider the infinite system and the finite systems, for each $N=1,2,\ldots$,
all with zero initial state (with all queues being $0$). Then,\\
(i) All corresponding processes can be coupled so that for all $t\ge 0$,
$$
Q^{(1)}(t) \le Q^{(2)}(t) \le \ldots Q^{(N)}(t) \le \ldots \le Q(t),
$$
$$
Q^{(1)}(t) \preceq Q^{(2)}(t) \preceq \ldots Q^{(N)}(t) \preceq \ldots \preceq Q(t).
$$
(ii) Process $Q(\cdot)$ is stochastically non-decreasing in $t$ (in the sense of $\le$ order),
and process $Q^{(N)}(\cdot)$ is stochastically non-decreasing in both $t$ and $N$.\\
(iii) We have convergences in distribution
\beql{eq-conv-stationary0}
Q^{(N)}(t) \Rightarrow Q^{(N)}(\infty), ~~\mbox{as}~t\to\infty,
\end{equation}
\beql{eq-conv-stationary1}
Q^{(N)}(\infty) \Rightarrow Q(\infty), ~~\mbox{as}~N\to\infty,
\end{equation}
where the sequence in the left-hand side is stochastically non-decreasing and $Q(\infty)$
just denotes its limit,
\beql{eq-conv-stationary2}
Q(t) \Rightarrow Q(\infty), ~~\mbox{as}~t\to\infty.
\end{equation}
(iv) The distribution of the limit $Q(\infty)$ is a stationary distribution 
(namely, the lower invariant measure)
of Markov process
$Q(\cdot)$.\\
(v) Condition \eqn{eq-bound-tail-inf} holds for all $t\ge 0$ and for the stationary
state $Q(\infty)$.
\end{prop}

Proposition~\ref{prop-ordering}(iii)
 of course implies that 
$Q_1^{(N)}(\infty)$ is stochastically non-decreasing, converging in distribution
to $Q_1(\infty)$ and therefore,
to prove Theorem~\ref{th-main},
we can study the distribution of $Q_1(\infty)$.
(If it happens that $Q_1(\infty)=\infty$, this implies $Q_1^{(N)}(\infty)\to\infty$
in probability as $N\to\infty$.)

We will need one more monotonicity property, which is also
a corollary of Lemma~\ref{lem1}.
Its meaning is very simple: if in addition
to process $Q(\cdot)$ we consider another process $Q'(\cdot)$ which is 
constructed the same way as $Q(\cdot)$, but with some additional
exclusions (``obstructions'') on the movement of the customers,
then $Q'(\cdot)$ will stay ``behind'' $Q(\cdot)$ in the sense
of $\preceq$ order.

\begin{prop}
\label{prop-ordering-under-obstruction}
Consider a fixed realization of the process $Q(\cdot)$, 
which includes a fixed finite initial state $Q(0)$ and realizations
of the input process (at site 1) and of all processes $\Pi_n$, $n\ge 1$.
Now, suppose further that in the realizations of
processes $\Pi_n$, 
some of the points (jumps) are marked as ``valid'' 
(in an arbitrary way) and the remaining points are ``invalid''.
Consider another realization $Q'(\cdot)$, with the same initial state $Q'(0)=Q(0)$,
and constructed in the same way as $Q(t)$, except invalid points of
processes $\Pi_n$ are ``ignored'' (cause no action). Then,
$$
Q'(t) \preceq Q(t), ~~\forall t\ge 0,
$$
and, in particular, $Q'_1(t) \ge Q_1(t)$ for all $t$.
\end{prop}

\section{Subcritical case: $\lambda <1/4$}
\label{sec-subcritical}

Suppose $\lambda <1/4$. We will construct 
a process $Q'(\cdot)$, coupled with $Q(\cdot)$ so that 
Proposition~\ref{prop-ordering-under-obstruction} holds path-wise,
and such that we can obtain a stochastic upper bound on $Q'_1(t)$.

Consider process $Q(\cdot)$, with zero initial state $\|Q(0)\|=0$,
constructed on the probability space 
described in Section~\ref{sec-basic}.
We will extend the probability space to define a stationary {\em totally asymmetric
simple exclusion process} (TASEP, cf. Chapter VIII of \cite{Liggett-book}),
with sites being integers $n\in \Z$, and particles moving
to the ``right''. Specifically, let us choose arbitrary (density) $\rho \le 1/2$,
such that 
$\mu=\rho(1-\rho) > \lambda$. 
Let $Y_n(t)\in \{0,1\}$ denote the number of particles of TASEP at site $n$ at time $t\ge 0$.
We augment the probability space 
so that, independently of all other driving processes, at time $0$
each site $n\in \Z$ contains a particle, $Y_n(0)=1$, with probability $\rho$
and does not contain one, $Y_n(0)=0$, with probability $1-\rho$.
The movement of TASEP particles will be driven by the same Poisson processes $\Pi_n$,
that drive process $Q(\cdot)$. 
(The exogenous input process at site $1$ does {\em not} affect TASEP.)
If time $\tau$ is a point (jump)
of $\Pi_n$ associated with site $n$, then the particle located
at $n$ (if any) attempts to jump to site $n+1$ -- it actually does jump
if site $n+1$ is empty, and it stays at $n$ otherwise. It is well known
that if the initial state $Y(0)$ has Benoulli distribution
as defined above, then the TASEP process $Y(\cdot)$ is stationary 
(cf. Theorem VIII.2.1 of \cite{Liggett-book}). The flux of this process, i.e. the average
departure rate of particles from a given site,
is $\mu=\rho(1-\rho)$; the average speed
of a given (``tagged'') particle is $v=1-\rho$. 
(For example, $\rho = 1/2$
gives the maximum possible flux $\mu=1/4 > \lambda$; this is where the condition
$\lambda<1/4$ comes from: $\lambda$ needs to be less than the flux of a stationary TASEP.)

The process $Q'(\cdot)$ has the same (zero) initial state as $Q(\cdot)$,
and is constructed the same way as $Q(\cdot)$ except for an additional
exclusion: a customer (particle) from queue (site) $1$ cannot move to 
queue $2$ at time $\tau$ unless a particle of the TASEP jumps from site $1$ to $2$
at $\tau$. We now record basic properties of process $Q'(\cdot)$.

\begin{prop}
\label{prop-bounding-process}
(i) $Q'(t) \preceq Q(t)$ and $Q'_1(t) \ge Q_1(t)$ for all $t\ge 0$ 
(by Proposition~\ref{prop-ordering-under-obstruction}).\\
(ii) For any $n\ge 2$ and any $t\ge 0$, $Q'_n(t) \le Y_n(t)$. In other words,
at all sites to the right of $1$, process $Q'(\cdot)$ stays ``within TASEP'';
in particular, there can be at most one particle in each site $n\ge 2$.\\
(iii) A particle jump from site $1$ to $2$ in the process $Q'(\cdot)$
happens at time $\tau$ if and only if $Q'_1(\tau-)\ge 1$
and there is a jump of TASEP particle from $1$ to $2$ at time $\tau$.\\
(iv) Process $Q'(\cdot)$ is stochastically non-decreasing with $t$.
\end{prop}

We know from Proposition~\ref{prop-ordering-under-obstruction}
 that $Q'_1(t)$ is an upper bound of $Q_1(t)$.
The behavior of queue length $Q'_1(t)$ is such that it is initially zero,
$Q'_1(0)=0$, the input process is Poisson with rate $\lambda$, and the ``service
process'' is the stationary process of TASEP particle jumps from site $1$ to $2$.
We denote by $A(t_1,t_2)$ and $S(t_1,t_2)$ the number of points (jumps) of the 
arrival and service processes, respectively,
in the interval $(t_1,t_2]$; WLOG we assume that these processes
are defined for all real times, that is $t_1,t_2\in \R$, $t_1 \le t_2$;
their average rates are $\lambda$ and $\mu=E S(t_1,t_2)/(t_2-t_1)$, respectively.
From the large deviations estimate given below in Lemma~\ref{low-ldp-service},
it also follows that $S(-s,0)/s \to \mu$, $s\to\infty$, with probability $1$.

From classical Loynes constructions \cite{loy} it is known that the 
distribution of $Q'_1(t)$ is stochastically non-decreasing with $t$,
and as $t\to \infty$ it weakly converges to the stationary distribution,
which in turn is same as that of random variable
\beql{eq-stationary-queue}
Q'_1(\infty) \doteq \sup_{s\ge 0} [A(-s,0)-S(-s,0)].
\end{equation}
$Q'_1(\infty)$ is a proper random variable due to condition $\lambda<\mu$,
which guarantees that the RHS of \eqn{eq-stationary-queue} is finite w.p.1.

Thus, we see that, as $t\to\infty$, $Q'_1(t)$ (and then $Q_1(t)$) 
remains stochastically bounded by $Q'_1(\infty)$. 
Moreover, the large deviations estimates (Lemma~\ref{low-ldp-service})
imply exponential bound on the tail of $Q'_1(\infty)$ distribution.
We proceed with the details.

The following fact is a known property of the stationary TASEP defined above. At time $0$, let us
consider the site with smallest index $n_0 \ge 2$ that contains a particle, and tag this particle.
(Obviously, $n_0-2$ has geometric distribution.) If we consider the point process
of jumps of tagged particle in time interval $[0,\infty)$, it is 
a Poisson process of rate $v=1-\rho$ (cf. Corollary VIII.4.9 of \cite{Liggett-book}).
Therefore, the location $H(t)$ of tagged particle
at time $t\ge 0$ is $H(t) = 2 + H_1 + H_2(t)$,
where $H_1$ is geometric random variable with mean $(1-\rho)/\rho$,
$H_2(t)$ is Poisson r.v. with mean $v t$, and $H_1$ and $H_2(t)$
are independent. From the stationarity of TASEP we also know that,
at any time $t$, the total number $G(n)$ of particles at sites $2,3, ...,n$
is simply the sum of $n-1$ independent Bernoulli variables with mean $\rho$.
Using ``separate'' large deviations estimates for $H(t)$ and $G(n)$,
even though these two r.v. are not independent, we obtain the following

\begin{lem}
\label{low-ldp-service}
For any $\delta>0$, there exist $C_3>0$ and $C_4>0$ such that
\beql{eq-ldp}
\pr\{|S(0,t) - \mu t| > \delta t\} \le C_3 e^{-C_4 t}.
\end{equation}
\end{lem}

\begin{proof}
To prove bound
$$ 
\pr\{S(0,t) < (\mu-\delta)t\} \le C_3 e^{-C_4 t}
$$ 
we can choose $\epsilon>0$, small enough, so that
$$
\pr\{S(0,t) < (\mu-\delta)t\} \le \pr\{H(t) \le (v-\epsilon)t\} 
+ \pr\{G((v-\epsilon)t) \le (\rho-\epsilon)(v-\epsilon)t\}.
$$
Bound 
$$ 
\pr\{S(0,t) > (\mu+\delta)t\} \le C_3 e^{-C_4 t}.
$$ 
is proved similarly.
\end{proof}

\begin{thm}
\label{queue-exp-tail}
Assume $\lambda<1/4$.
There exist $C_1>0$ and $C_2>0$ such that
\beql{eq-th-tail}
\pr\{Q'_1(\infty) > r\} \le C_1 e^{-C_2 r};
\end{equation}
and then
$$
\pr\{Q_1(\infty) > r\} \le C_1 e^{-C_2 r}.
$$
\end{thm}

{\bf Remark.} Given the large deviations bound \eqn{eq-ldp}, the argument 
to prove \eqn{eq-th-tail} is quite standard (see \cite{GW1994,Duff1995}).
However, formally, \cite{Duff1995} for example, requires a stronger condition,
large deviations principle (LDP) for $S(0,t)/t$; 
we did not find this LDP result in the literature and it is not needed for our purposes.
Hence, for completeness, we give a proof of the theorem.

\begin{proof}
It follows from the definition \eqn{eq-stationary-queue} that for any fixed $d>0$
\beql{eq-bound-discrete}
Q'_1(\infty) \le \sup_{k=1,2,\ldots} [A(-kd,0)-S(-(k-1)d,0)] 
\equiv A(-d,0) + \sup_{k=1,2,\ldots} [A(-kd,-d)-S(-(k-1)d,0)]
\end{equation}
Let us fix $b>0$ such that $b\lambda < 1$, and for each $r>0$ we will choose
$d=br$. Then we can write
$$
\pr\{Q'_1(\infty) > r\} \le \pr\{ A(-br,0) \ge r\}
+ \sum_{k=2,3,\ldots} \pr\{A(-kd,-d)-S(-(k-1)d,0) \ge 0\}.
$$
If we fix $\delta>0$ small enough so that $\lambda+\delta < \mu - \delta$
and $(\lambda + \delta)b <1$,
we have
$$
\pr\{Q'_1(\infty) > r\} \le \pr\{ A(-br,0) \ge (\lambda + \delta) b r\} +
$$
$$
\sum_{k=2,3,\ldots} [\pr\{A(-kd,-d) \ge (\lambda + \delta)(k-1)d\}
   + \pr\{S(-(k-1)d,0) \le (\mu - \delta)(k-1)d\}].
$$
It remains to apply large deviations bounds on $A(\cdot)$ and $S(\cdot)$
(Lemma~\ref{low-ldp-service}).
\end{proof}

Let us define 
$$
B(Q(t)) \doteq \min\{n=1,2\ldots ~|~ Q_n(t)=0\},
$$
$$
B'(Q(t)) \doteq \min\{n=1,2\ldots ~|~ \sum_{k=1}^n Q_k(t) < n\}.
$$
$B(Q(t))$ is the ``busy interval'' of the process $Q(t)$: to the right of, and including,
site $B(Q(t))$ all sites have at most one customer and so the (instantaneous)
evolution of the process follows the same ``rules'' as that of TASEP.
The interpretation of $B'(Q(t))$ is as follows: starting from the state
$Q(t)$, we take a customer from the left-most site with 2 or more customers,
and move it to the left-most empty site, and then repeat until all sites
have at most 1 customer; then site $B'(Q(t))$ is the left-most empty site
of the modified state.

Obviously, $B'(Q(t)) \ge B(Q(t))$; and it is easy to check that
$Q'(t) \preceq Q(t)$ implies $B'(Q'(t)) \ge B'(Q(t))$. Thus,
$B'(Q'(t)) \ge B(Q(t))$. 

\begin{thm}
\label{busy-interval-exp-tail}
Assume $\lambda<1/4$.
There exist $C_5>0$ and $C_6>0$ such that, for all $t\ge 0$
$$
\pr\{B'(Q'(t)) > n\} \le C_5 e^{-C_6 n},
$$
and then
$$
\pr\{B(Q(t)) > n\} \le C_5 e^{-C_6 n}.
$$
\end{thm}

\begin{proof}
The result is easily derived from the following facts: $Q'_1(t)$ is stochastically bounded by
$Q'_1(\infty)$;  the distribution of $Q'_1(\infty)$ has exponentially decaying tail
(Theorem~\ref{queue-exp-tail});
process $Q'(t)$ ``stays within'' stationary TASEP of density $\rho$
at all sites $n\ge 2$;
definition of $B'$.
\end{proof}

Theorem~\ref{busy-interval-exp-tail} illustrates in particular the fact that
(when $\lambda<1/4$) the infinite-tandem system under BP algorithm in stationary regime
 is such that there only a ``small''  number of sites
(from site 1 to site $B(Q(t))$) where queue can be greater than 1; while all sites to the
right of $B(Q(t))$ have queue of at most one, and therefore the behavior of the process
``to the right of $B(Q(t))$'' is same as that of TASEP.

\section{Supercritical case: $\lambda > 1/4$}
\label{sec-supercritical}

Note that if $\lambda>1/2$ we immediately see from \eqn{eq-mean-growth}
that $\E Q_1(\infty) = \lim_N \E Q_1^{(N)}(\infty) = \infty$.
Here we prove that, in fact, $Q_1(\infty)$ is infinite w.p.1, under a weaker condition $\lambda>1/4$.
The intuition behind our argument is as follows.
Unless $Q_1(\infty) = \infty$ w.p.1, busy interval $B(Q(t))$ must be stochastically
bounded w.p.1. Then, in stationary regime, all sites ``far enough'' to the right have
at most one customer (particle) in them, and therefore the process ``there'' behaves
as TASEP. However, a stationary TASEP cannot have flux greater than $1/4$, while the flux
of our process must be $\lambda>1/4$, a contradiction.

\begin{thm}
\label{queue-unbounded}
Assume $\lambda>1/4$. Then $Q_1(\infty) = \infty$. (And then $Q_n(\infty) = \infty$
for all $n\ge 1$.)
\end{thm}

\begin{proof} The proof is by contradiction -- assume $Q_1(\infty)$ is finite
with positive probability. It is easy to see that this is possible only
if $Q_1(\infty)<\infty$ with probability $1$.
(Otherwise, using the facts that $Q_1(t)$ is stochastically increasing and
weakly converges to $Q_1(\infty)$, and using coupling, we could show that
$Q_1(\infty)=\infty$ with prob. $1$.)
This in turn implies that $B(Q(\infty))$ (which is the weak limit
of stochastically non-decreasing r.v. $B(Q(t))$) is also finite w.p.1.
(Otherwise, we again could show that $Q_1(\infty)=\infty$ with prob. $1$.)

The stationary version of the process $Q(\cdot)$ (i.e. the one with
stationary distribution equal to that of $Q(\infty)$)
we denote by $\tilde{Q}(\cdot)$. This process is such that: w.p.1 
condition \eqn{eq-bound-tail-inf} holds for all $t\ge 0$,
and therefore all $\tilde{Q}_n(t)$ for all $t$ are uniformly
stochastically upper bounded by  $Q_1(\infty)+1$ and then finite w.p.1;
$B(\tilde{Q}(t))$ are finite w.p.1 (equally distributed) random variables
for all $t$; the flux is equal to $\lambda$, namely, $\E F_n(t)/t=\lambda$
for any $t>0$ and $n\ge 1$, where $F_n(t)$ is the number of customers arrivals
at site $n$ in interval $(0,t]$. Note that for $n\ge 2$, the arrival process $F_n(\cdot)$ 
at site $n$ is 
the departure process from site $n-1$. This implies
that increments of processes $F_n(\cdot)$, $n\ge 2$, are stochastically upper bounded 
by the increments of independent Poisson processes $\Pi_{n-1}(\cdot)$.

Consider space-shifted processes 
$\{[T_m \tilde{Q}](\cdot), [T_m F](\cdot), [T_m \Pi](\cdot)\},$
$m=1,2,\ldots$, where $[T_m \tilde{Q}]_i(t) = \tilde{Q}_{m+i}(t)$, $i=1,2,\ldots$,
$t\ge 0$, and $[T_m F]_i(t)$ and  $[T_m \Pi]_i(t)$ defined similarly.
For each $m$ this process is such that $[T_m \tilde{Q}](\cdot)]$ and the increments
of $[T_m F](\cdot)$ and $[T_m \Pi](\cdot)$ are stationary.
(Note that process $[T_m \tilde{Q}](\cdot)$ is {\em not} Markov.)
This process is still well-defined if we assume that each component
$[T_m \tilde{Q}]_i(t)$, $[T_m F]_i(t)$ and $[T_m \Pi]_i(t)$ takes values
in (non-compact) space $\Z_+$ with the usual topology
(because we know that they are finite w.p.1), and with corresponding product topology 
in the space of values of the entire process. Note that, since 
 $B(\tilde{Q}(t))$ is finite w.p.1 and condition \eqn{eq-bound-tail-inf} holds,
we have
\beql{eq-queue-one}
\lim_{m\to\infty} \pr\{\sup_{n\ge 1} [T_m \tilde{Q}]_n(t) \le 1\} =1, ~~\forall t\ge 0.
\end{equation}
Then, using properties of process $\tilde{Q}(\cdot)$ described earlier,
in particular the fact that
increments of $F_n(\cdot)$ are bounded by those of $\Pi_{n-1}(\cdot)$,
it is easy to see that a process consisting of any finite subset
of components $[T_m \tilde{Q}]_i(\cdot)$, $[T_m F]_i(\cdot)$ and $[T_m \Pi]_i(\cdot)$
is tight (cf. Theorem 15.6 in \cite{bil}).
Consequently,
there exists
a subsequence of $\{m\}$ along which the shifted process converges in distribution
to a process 
$\{\overline{Q}(\cdot),  \overline{F}(\cdot), \overline{\Pi}(\cdot)\}$, 
which has the following 
(easily verifiable) structure and properties:\\
(a) $\overline{Q}(\cdot)$ is stationary, with flux equal $\lambda$;\\
(b) $\overline{Q}_n(t) \le 1$ for all $n$ and all $t$;\\
(c) the movement of customers between sites is driven by independent,
unit rate Poisson processes $\overline{\Pi}_i(\cdot)$, according to BP algorithm rules;\\
(d) By (b) and (c), $\overline{Q}(\cdot)$ is a TASEP (with the ``exogenous'' arrivals
at site $1$ forming a  stationary process, 
{\em not} independent of the ``rest of the process'').

Consider the following projection of process $\overline{Q}(\cdot)$. 
All particles arriving at site $1$ after time $0$ and the particle located at site $1$ 
at time $0$ (if any), we will call ``new'' particles, while all particles initially present
at sites $n\ge 2$ are ``old''. Let $Q^*(\cdot)$ denote the process ``keeping track'' of
new particles in $\overline{Q}(\cdot)$, namely $Q^*_n(t)=1$ if there is a new particle located
at site $n$ at time $t$, and $Q^*_n(t)=0$ otherwise. 
The flux of process $Q^*(\cdot)$ from site 1 to site 2 is obviously equal to the 
flux of $\overline{Q}(\cdot)$, which is $\lambda$.
We will compare $Q^*(\cdot)$ to the following TASEP $Q^{**}(\cdot)$,
coupled to it -- with the same Poisson processes driving movement between sites.
The initial state of $Q^{**}(\cdot)$ is: $Q^{**}_1(0)=1$ and $Q^{**}_n(0)=0$ for $n\ge 2$.
By definition, $Q^{**}_1(t)\equiv 1$, that is if at any time a particle moves
from site $1$ to $2$, it is immediately replaced at site $1$ by another particle.
Using the path-wise monotonicity properties given in Section~\ref{sec-basic},
it is easy to see that $Q^*(t) \preceq Q^{**}(t)$, which implies
$$
F_2^*(t) \le F_2^{**}(t), ~~t\ge 0,
$$
where $F_2^*(t)$ and $F_2^{**}(t)$ are the numbers of particle arrivals 
in $(0,t]$ at site $2$
in the processes $Q^*(\cdot)$ and $Q^{**}(\cdot)$, respectively. Then,
\beql{eq-flux-order}
\liminf_{t\to\infty} \frac{\E F_2^{**}(t)}{t} \ge \lim_{t\to\infty} \frac{\E F_2^{*}(t)}{t} = \lambda.
\end{equation}
The TASEP $Q^{**}(\cdot)$ is a special case of one of the processes studied in \cite{Liggett75}.
It is known (see Theorem 1.8(a) and Theorem 1.7(b) of \cite{Liggett75}) that 
the distribution of $Q^{**}(t)$ converges to a stationary distribution, with the corresponding
stationary process having flux $1/4$. This means that
$\lim_{t\to\infty} \E F_2^{**}(t)/t = 1/4$, which contradicts \eqn{eq-flux-order}
since $\lambda>1/4$. Proof is complete.
\end{proof}

\section{Remark on more general input processes}
\label{sec-gen-input}

The Poisson assumption on the input process is adopted to simplify exposition.
We belive our main results can be easily generalized for the case of a stationary ergodic 
input process $A(\cdot)$, as long as large deviations bound
\beql{eq-ldp555}
\pr\{|A(0,t) - \lambda t| > \delta t\} \le C_3 e^{-C_4 t},
\end{equation}
analogous to \eqn{eq-ldp}, holds for any $\delta>0$. Moreover, if \eqn{eq-ldp555} does not hold,
and we only have the ergodicity of $A(\cdot)$, the uniform stochastic boundedness results of
Theorem~\ref{queue-exp-tail} (and then Theorem~\ref{th-main}(i)) 
and Theorem \ref{busy-interval-exp-tail} will still hold,
except the  bounds are proper (finite w.p.1) random variables, not necessarily with exponential tails.

\noindent
{\bf Acknowledgement.} I would like to thank 
Yuliy Baryshnikov for extremely helpful discussions throughout the course
of this work.

\end{document}